\documentclass{amsart}

\usepackage{amssymb}

\usepackage{enumerate}

\usepackage{graphicx}

\newtheorem{thm}{Theorem}[section]
\newtheorem{cor}[thm]{Corollary}

\newtheorem{prop}[thm]{Proposition}

\theoremstyle{definition}
\newtheorem{defn}[thm]{Definition}
\newtheorem{rem}[thm]{Remark}
\newtheorem{exa}[thm]{Example}

\usepackage[utf8]{inputenc}
\usepackage[T1]{fontenc}

\usepackage{xcolor}
\usepackage{graphics,graphicx}
\usepackage{tikz}
\usetikzlibrary{patterns}

\begin{document}

\title{Alpha-Invariant of toric line bundles}

\author[T. Delcroix]{Thibaut Delcroix}
\address{Univ. Grenoble Alpes, IF, F-38000 Grenoble,France \\
CNRS, IF, F-38000 Grenoble,France}
\email{thibaut.delcroix@ujf-grenoble.fr}

\date{}

\begin{abstract}
We generalize the work of Jian Song to compute the $\alpha$-invariant of any (nef and big) toric 
line bundle in terms of the associated polytope. We use the analytic version of the computation of the log 
canonical threshold of monomial ideals to give the log canonical threshold of any non-negatively curved 
singular hermitian metric on the line bundle, and deduce the $\alpha$-invariant from this.
\end{abstract}

\maketitle

\section*{Introduction}

The $\alpha$-invariant of a line bundle $L$ on a complex manifold $X$ 
 is an invariant measuring the singularities of the non-negatively curved singular hermitian metrics 
on $L$. It was introduced by Tian in the case of the anticanonical bundle on a Fano manifold. 
Tian showed in \cite{Tia87}
that if the $\alpha$-invariant of the anticanonical bundle is strictly greater 
than $\frac{n}{n+1}$, then the Fano manifold admits a Kähler-Einstein metric.

The Yau-Tian-Donaldson conjecture asserts in general that $X$ admits an extremal metric in $c_1(L)$
if and only if the line bundle $L$ is K-stable. 
It was proved in \cite{CDS15a,CDS15b,CDS15c,Tia12}
that it holds when $L$ is the anticanonical bundle.
In particular (as it was shown also in \cite{OS12}), 
if the $\alpha$-invariant of the anticanonical bundle is greater than 
$\frac{n}{n+1}$, then the anticanonical bundle is K-stable.
Dervan \cite{Der13}
gave a similar condition of K-stability for a general line bundle, involving again its $\alpha$-invariant.
This is one motivation to compute explicitly the $\alpha$-invariants of line bundles when possible.

In \cite{CSD08}, Chel'tsov and Shramov 
computed for example the $\alpha$-invariant of the anticanonical bundle for many Fano manifolds of
dimension three.
In higher dimensions, Song \cite{Son05}
proved a formula giving the $\alpha$-invariant of the anticanonical bundle on a toric Fano manifold
in terms of its polytope. The only toric manifolds satisfying Tian's criterion are the symmetric 
toric manifolds. Batyrev and Selivanova \cite{BS99} proved first that their $\alpha$-invariant was one,
so that they admit a Kähler-Einstein metric.
Wang and Zhu \cite{WZ04}
fully settled the question of the existence of Kähler-Einstein metrics on toric Fano manifolds, 
and an illustration that Tian's criterion is only a sufficient condition can be found in the toric world \cite{NP11}.

The $\alpha$-invariant of a line bundle $L$ is strongly related to the log canonical thresholds (lct)
of metrics on $L$. The log canonical threshold was initially an algebraic invariant defined for ideal sheaves,
but it was shown to coincide with the complex singularity exponent and Demailly defines the log canonical 
threshold of any non-negatively curved singular hermitian metric on a line bundle in \cite{CSD08} for example.

One of the main examples of computation of log canonical threshold is in the case of monomial ideals.
Howald carried out the computation of the lct of such an ideal in terms of its Newton polygon \cite{How01}.
One can find in Guenancia \cite{Gue12} an analytic proof of this result, generalized to compute the lct 
of an ideal generated by a "toric" psh function on a neighborhood of $0\in \mathbb{C}^n$, i.e. a function 
invariant under rotation in each coordinate.

Since the only smooth affine toric manifolds without torus factor are isomorphic to $\mathbb{C}^n$, 
the computation of Guenancia in fact gives the log canonical threshold of any invariant metric on an 
affine smooth toric manifold, as we explain in Section~\ref{seclct}.

In this note, we give a formula for the $\alpha$-invariant of any line bundle $L$ on a compact smooth 
toric manifold in terms of its polytope. We also compute the log canonical threshold of any 
invariant non-negatively curved singular metric on $L$.

After this article was accepted, the author was informed that other authors computed similar invariants using other methods (H. Li, Y. Shi, Y. Yao \cite{LSY15}, and F. Ambro \cite{Amb14}).

\section{Line bundles on smooth toric manifolds}

\subsection{Toric manifolds}

Let us recall some basic facts about toric varieties (see \cite{Ful93}, \cite{Oda88}, \cite{CLS11}).

Let $T=(\mathbb{C}^*)^n$ be an algebraic torus.
Denote its group of characters by $M$, which is isomorphic to $\mathbb{Z}^n$ through the choice of a basis, 
and let $M_{\mathbb{R}}:=M\otimes \mathbb{R}\simeq \mathbb{R}^n$.
The dual $N$ of $M$ consists of the one parameter subgroups of $T$, and we let 
also $N_{\mathbb{R}}:=N\otimes \mathbb{R}\simeq \mathbb{R}^n$.

We denote by $T_c\simeq (S^1)^n$ the compact torus in $T$.

Considering only cones for the toric setting, we will call $\sigma \subset N_{\mathbb{R}}$ a \emph{cone}
if $\sigma$ is a convex cone generated by a finite set of elements of $N$.
The \emph{dual cone} $\sigma^{\vee}$ is defined as 
$$\sigma^{\vee}=\{x\in M_{\mathbb{R}}| \langle x,y \rangle \geq 0 ~ \forall y\in \sigma\}.$$

A \emph{fan} $\Sigma$ consists of a finite collection of cones $\sigma \subset N_{\mathbb{R}}$ such that 
every cone is \emph{strongly convex} (\emph{i.e.} $\{0\}$ is a face of $\sigma$), the faces of cones in $\Sigma$ 
are in $\Sigma$ and the intersection of two cones in $\Sigma$ is a union of faces of both.
The support of $\Sigma$ is $|\Sigma|:=\bigcup_{\sigma\in \Sigma} \sigma\subset N_{\mathbb{R}}$.

Recall that a fan $\Sigma$ in $N_{\mathbb{R}}$ determines a toric variety $X_{\Sigma}$, that is, a normal 
$T$-variety with an open and dense orbit isomorphic to $T$, and every toric variety
is obtained this way.

By the orbit-cone correspondence \cite[Theorem 3.2.6]{CLS11}, a maximal cone $\sigma$ of $\Sigma$ corresponds 
to a fixed point $z_{\sigma}$ in $X_{\Sigma}$. Also, a one-dimensional cone $\rho$ 
in $\Sigma$ corresponds
to a prime invariant divisor $D_{\rho}$ of $X_{\Sigma}$, and these divisors generate the group 
of Weil divisors of $X_{\Sigma}$. 
Let $\rho$ be such a cone, then we denote by $u_{\rho}$ the primitive vector in $N$ 
generating this ray.
We will denote by $\Sigma(r)$ the set of $r$-dimensional cones in $\Sigma$.

Many properties of $X_{\Sigma}$ can be read off from the fan. For example, $X_{\Sigma}$ is smooth if and 
only if every cone in the fan $\Sigma$ is generated by part of a basis of $N$. 
We will call a cone \emph{smooth} if it satisfies this condition.
The variety $X_{\Sigma}$ is complete if and only if $|\Sigma|=N_{\mathbb{R}}$.

We will assume in general in the following that either $|\Sigma|=N_{\mathbb{R}}$ or that 
$\Sigma$ is given by a strongly convex, full dimensional cone $\sigma$ and its faces, in which case we will 
denote $X_{\sigma}$ the corresponding (affine) toric variety.

\subsection{Line bundles}

Recall that a line bundle $L$ on a $G$-variety $X$ is called \emph{linearized} if there is an action of $G$ on 
$L$ such that for any $g\in G$ and $x\in X$, $g$ sends the fiber $L_x$ to the fiber 
$L_{g\cdot x}$ and the map defined this way between $L_x$ and $L_{g\cdot x}$ is linear.

To a $T$-linearized line bundle $L$ on $X_{\Sigma}$ is associated a set of characters $v_{\sigma}$, 
for $\sigma \in \Sigma(n)$. We define $v_{\sigma}$ as the opposite of the 
character of the action of $T$ on the fiber over the fixed point $z_{\sigma}$.

This defines the support function $g_L$ of $L$, which is a function on the support $|\Sigma|$
of $\Sigma$, linear on each cone, which takes integral values at points of $N$,
by 
$x\mapsto \langle v_{\sigma},x \rangle$ for $x\in \sigma$.

Another equivalent 
 data is the Weil divisor $D_L$ associated to $L$, which is related to $g_L$ by the following:
$D_L=-\sum_{\rho}g_L(u_{\rho})D_{\rho}$.

If $L$ is effective, then to $L$ is associated a polytope $P_L$ in $M_{\mathbb{R}}$.
This polytope can be defined as 
$$P_L=\left\{m\in M_{\mathbb{R}}| g_L(x)\leq  \langle m,x \rangle ~ \forall x\in |\Sigma|\right\}.$$

The properties of the line bundle can be read off from the polytope or the support function. 
In particular, we can associate to each point of $P_L\cap M$ a global section of $L$, and 
the collection of these sections form a basis of the space of algebraic sections of $L$.
Recall also the following,
where we assume that $|\Sigma|=N_{\mathbb{R}}$.

\begin{prop} 
\cite[Theorem 6.1.7]{CLS11}
The following are equivalent:
\begin{itemize}
\item $L$ is nef 
\item $L$ is generated by global sections
\item $\{v_{\sigma}\}$ is the set of vertices of $P_L$
\item $g_L$ is concave.
\end{itemize}
\end{prop}

\begin{prop}
\cite[Lemma 9.3.9]{CLS11}
$L$ is big iff $P_L$ has nonempty interior.
\end{prop}

\begin{prop}
\cite[Lemma 6.1.13]{CLS11}
The line bundle $L$ is ample iff $g_L$ is concave and $v_{\sigma}\neq v_{\sigma'}$ whenever
$\sigma\neq \sigma' \in \Sigma(n)$.
\end{prop}

\begin{exa}
\label{anticanonical}
The anticanonical divisor $-K_{X_{\Sigma}}$ on a toric manifold is given by 
$-K_{X_{\Sigma}}=\sum_{\rho}D_{\rho}$. 
It is always big on a toric manifold. 
\end{exa}

\subsection{Non-negatively curved singular metrics on line bundles}

\subsubsection{Potential on the torus}

Let $L$ be a $T$-linearized line bundle on $X_{\Sigma}$. 

Recall that any linearized line bundle on $T\simeq (\mathbb{C}^*)^n$ is trivial. 
Fix an invariant trivialization $s$ of $L$ on $T$.  

Given a hermitian metric $h$ on the line bundle $L$, we denote by $\varphi_h$ the local 
potential of $h$ on $T$, which is the function on $T$ defined by:
$$\varphi_h(z):=-\ln(||s(z)||_h).$$

The local potentials of a smooth hermitian metric are smooth. We will work here with singular 
metrics, whose local potential are \emph{a priori} only in $L^1_{\mathrm{loc}}$.
A singular hermitian metric $h$ is said to have non negative curvature (in the 
sense of currents) if and only if every local potential of $h$ is a psh function.

A $T_c$-invariant function $\varphi$ on $T$ is determined by a function $f$ on $N_{\mathbb{R}}$, 
identified with the Lie algebra of $T_c$, through the equivariant isomorphism:
$$T_c\times N_{\mathbb{R}} \longrightarrow T; ~((e^{i\theta_j})_j,(x_j)_j)\mapsto (e^{x_j+i\theta_j})_j.$$
Furthermore, $\varphi$ is psh if and only if $f$ is convex.

So to a non negatively curved, $T_c$-invariant metric $h$ on $L$ is associated a convex 
function $f_h$, which is the function on $N_{\mathbb{R}}$ determined by $\varphi_h$.

\subsubsection{Behavior at infinity of the potentials}

\begin{defn}
Let $L$ be a nef line bundle on $X_{\Sigma}$. 
The function $f_L: x\mapsto -g_L(-x)$ is a convex function on $N_{\mathbb{R}}$, and it is 
the potential of a continuous, $T_c$-invariant, non negatively curved metric on $L$ called 
the Batyrev-Tschinkel metric (see \cite{Mai00}), which we denote by $h_L$. 
\end{defn}

\begin{prop}
\label{inf}
The map $h\mapsto f_h$ defines a bijection between the singular hermitian $T_c$-invariant metrics on $L$ 
with non-negative curvature, and the convex functions on 
$N_{\mathbb{R}}$, such that there exists a constant $C$ with $f_h\leq f_L + C$ on $N_{\mathbb{R}}$.
\end{prop}

\begin{proof}
See also \cite[Proposition 3.3]{BB13}. 
Let $h$ be a singular hermitian $T_c$-invariant metrics on $L$ 
with non-negative curvature. Write $h=e^{-v}h_L$, and let $\omega_L$ be the curvature 
current of $h_L$. Then $v$ is a $\omega_L$-psh function on $X$.
In particular, $v$ is bounded from above on $X$. Denote by $u$ the convex function 
on $\mathbb{R}^n$ associated to the $T_c$-invariant function $v|_T$. Then we see that $f_h(x)-f_L(x)=u(x)$ is 
bounded above on $N_{\mathbb{R}}$.

Conversely, the standard fact that a psh function, which is bounded from above,
extends uniquely over an analytic set, allows one to extend $u:=f-f_L$ to an $\omega_L$-psh 
function on the whole of $X$ if $f$ satisfies the condition of the proposition.  
\end{proof}

\section{Log canonical thresholds}
\label{seclct}

\subsection{Definition}

Let $X$ be a compact complex manifold, and $L$ a line bundle on $X$.
Let $h$ be a singular hermitian metric on $L$.
We recall the definition of the log canonical threshold of $h$ 
(see the appendix of \cite{CSD08}).

\begin{defn} 
Let  $z\in X$.
The complex singularity exponent $c_z(h)$ of $h$ at $z$ is the 
supremum of the real $c>0$ such that $e^{-2c\varphi}$ is integrable 
in a neighborhood of $z$, where $\varphi$ is a local potential of 
$h$ near $z$.
\end{defn}

\begin{defn}
The log canonical threshold $\mathrm{lct}(h)$ of $h$ is defined as 
$$\mathrm{lct}(h)=\mathrm{inf}_{z\in X}c_z(h).$$
\end{defn}

\subsection{Newton body of a function}

\begin{defn}
Let $\sigma$ be a cone. 
Let $f$ be a function defined on $N_{\mathbb{R}}$.
Define the Newton body of $f$ on $\sigma$ as 
$$N_{\sigma}(f)=\{m\in M_{\mathbb{R}}; f(x)-\langle m,x \rangle \geq O(1), ~ \forall x\in \sigma\}.$$ 
\end{defn}

If $\sigma=N_{\mathbb{R}}$ we will write $N(f)$.

The following properties of the Newton body will be useful. 

\begin{prop} 
\label{propN}
For any function $f$, $N_{\sigma}(f)$ is convex, and 
$$N_{\sigma}(f)=N_{\sigma}(f)-\sigma^{\vee}.$$ 
If $f$ is convex, then for any $y\in N_{\mathbb{R}}$, 
$$N_{\sigma}(f)=\{m\in M_{\mathbb{R}}; f(t)-\langle m,t \rangle \geq O(1), ~ \forall t\in y+\sigma\}.$$
\end{prop}

\begin{proof}
The first two properties are trivial. Let us briefly prove the last statement.

Let $m$ be in the right-hand set, \emph{i.e.} $\{f(t)-\langle m,t \rangle \geq O(1) ~ \forall t\in y+\sigma\}$.
Let $x=t-y\in \sigma$ for $t\in y+\sigma$.
By convexity, $f(x+y)\leq \frac{1}{2}(f(2x)+f(2y))$ so we get
$$f(2x)\geq 2f(x+y)-f(2y)=2f(t)-f(2y)$$
Subtracting $\langle m,2x \rangle $ gives 
$$f(2x)-\langle m,2x \rangle \geq 2(f(t)-\langle m,t \rangle )+(2\langle m,y \rangle -f(2y)).$$
The right hand side is the sum of a lower-bounded function of $t\in y+\sigma$ and a constant, 
so the left hand side is a lower-bounded function of $x\in \sigma$.

This shows one inclusion and the other is proved by a similar argument.
\end{proof}

Given a non negatively curved $T_c$-invariant metric $h$ on $L$, we define the 
associated convex subset $P_h$ of $M_{\mathbb{R}}$, as the Newton body of $f_h$.

\begin{prop}
\label{Polytope} \mbox{}
\begin{itemize}
\item For the Batyrev-Tschinkel metric $h_L$, we recover the polytope $P_L$.
\item For any $T_c$-invariant, non-negatively curved metric $h$ on $L$, $P_h\subset P_L$.
\item If $h$ is smooth, we also have $P_h=P_L$
\end{itemize}
\end{prop}

\begin{proof}
For the first statement, observe that $m\in P_L$ if and only if for any cone $\sigma \in \Sigma$,
for all $x\in \sigma$, $g_L(x)=\langle v_{\sigma},x \rangle \leq ~\langle m,x \rangle$.
This inequality is equivalent to $-\langle v_{\sigma},x \rangle+\langle m,x \rangle \geq 0$ and since the functions involved 
are linear, it is satisfied for all $x\in \sigma$ if and only if 
$-\langle v_{\sigma},x \rangle+\langle m,x \rangle$ is bounded below on $\sigma$.
Since $f_L(-x)=-g_L(x)=-\langle v_{\sigma},x \rangle$ for $x\in \sigma$,  we get that $m\in P_L$ if and only
if for every cone $\sigma \in  \Sigma$, the function $f_L(-x)-\langle m,-x \rangle$ is bounded below 
on $\sigma$. 
Finally, this can be translated as: for every cone $\sigma \in  \Sigma$, the function $f_L(y)-\langle m,y \rangle$ is bounded below 
on $-\sigma$. 
To conclude, we note that $N(f_L)=\bigcap_{\sigma}N_{-\sigma}(f_L)$.

The second statement is an easy consequence of the first and Proposition~\ref{inf}
since whenever two functions $f$ and $g$ satisfy $f\leq g+C$ for a constant $C$, we have 
trivially $N_{\sigma}(f)\subset N_{\sigma}(g)$.

For the last statement, remark that in this case, $f_h-f_L$ extends to a continuous function on $X_{\Sigma}$,
so we have $f_L-C\leq f_h \leq f_L+C$ for some constant $C$. The same property of Newton bodies allows one to conclude.
\end{proof}

\subsection{Integrability condition}

The first result on log canonical thresholds on toric varieties was the computation by Howald \cite{How01}
in the case of monomial ideals. Guenancia gave an analytic proof of this result, extending the 
computation to the case of non algebraic psh functions. The key ingredient in this analytic 
version is the following integrability condition.

\begin{prop}
\label{Gue}
 (see \cite{Gue12})
Let $\sigma$ be a smooth cone of maximum dimension.
Let $f$ be a convex function on $N_{\mathbb{R}}$. Then $e^{-f}$ is integrable on all 
translates of $\sigma$ if and only if $0\in \mathrm{Int}(N_{\sigma}(f))$.
\end{prop}

This is essentially the result in Guénancia \cite{Gue12}
because any smooth affine toric manifold with no torus factor is isomorphic to $\mathbb{C}^n$.
However we describe the change of variables used precisely, to use it later in the compact case.

\begin{proof}
Choose a basis of $N$ formed by the generators of the extremal rays of $\sigma$, then define
$S_{\sigma}$ to be the isomorphism from $N$ to $\mathbb{Z}^n$ sending the chosen basis to 
the canonical basis of $\mathbb{Z}^n$.

Let $f$ be a function on $N_{\mathbb{R}}$, and $g$ the function on $\mathbb{R}^n$ such that 
$f=g\circ S_{\sigma}$. Then from the definition of Newton body we have 
$N_{\sigma}(f)=S_{\sigma}^*(N_D(g))$, where $S_{\sigma}^*$ is the dual isomorphism 
from $\mathbb{Z}^n$ to $M$ and $D$ is the cone generated by the canonical basis  of 
$\mathbb{Z}^n$.

Using the change of variables, $e^{-f}$ is integrable on all translates of $\sigma$ 
if and only if $e^{-f\circ S_{\sigma}^{-1}}$ is integrable on all translates of $D$.
Apply \cite[Proposition 1.9]{Gue12} to the concave function $-f\circ S_{\sigma}^{-1}$.
This proves that we have integrability if and only if 
$0\in \mathrm{Int}(N_D(f\circ S_{\sigma}^{-1}))$.
Using $S_{\sigma}^*$, which is linear, this indeed translates to 
$0\in \mathrm{Int}(N_{\sigma}(f))$.

Remark that the statement in \cite[Proposition 1.9]{Gue12} only mentions integrability on 
$D$, but the equivalence with integrability on all translates is easily derived from
Proposition~\ref{propN}.
\end{proof}

\subsection{lct on an affine smooth toric manifold}

\begin{prop}
\label{lctaff}
Let $\sigma$ be a smooth cone of maximum dimension, $X_{\sigma}$ the corresponding 
smooth affine toric manifold. Let $L$ be a linearized line bundle on $X_{\sigma}$, and 
$h$ a $T_c$-invariant metric with non-negative curvature. Then 
$$\mathrm{lct}(h)=\mathrm{sup}\{c>0| cv_{\sigma} \in \mathrm{Int}(N_{-\sigma}(cf_h))-S_{\sigma}^*(1,\ldots,1)\}.$$  
\end{prop}

\begin{proof}
The change of variables for cones $S_{\sigma}$ in the proof of Proposition~\ref{Gue} gives (by \cite[Theorem 3.3.4]{CLS11})
an equivariant isomorphism between $X_{\sigma}$ and $\mathbb{C}^n$, which we denote again by $S_{\sigma}$.

Any linearized line 
bundle on $\mathbb{C}^n$ is trivial, so $L$ admits a global equivariant trivialization $t$ on $X_{\sigma}$. 
Remark that, at the fixed point $z_{\sigma}$, we have $g\cdot t(z_{\sigma}) = -v_{\sigma}(t(z_{\sigma}))$
by definition of $v_{\sigma}$. Restricting to $T$ and remembering that $s$ is an invariant trivialization
of $L$ on $T$, we deduce that up to renormalization by a constant, $t(z)=v_{\sigma}(z)s(z)$ on $T$.

We can now look at the potential $\psi$ of $h$ with respect to the trivialization $t$, and remark that, 
on $T$, and if $\varphi$ denotes the potential of $h$ with respect to $s$ on $T$, we have 
$\psi(z)=\langle -v_{\sigma},\ln|z| \rangle + \varphi(z)$. 

Let $y\in N_{\mathbb{R}}$. Using again the isomorphism $T_c\times N_{\mathbb{R}}\simeq T$, we consider
$T_c \times (y-\sigma)$ as a subset of $T$, and denote by $C_y$ the closure of this set in $X_{\sigma}$.
Each set $C_y$ is a neighborhood of $z_{\sigma}$ in $X_{\sigma}$, and they form a basis of neighborhoods. 
Observe that the collection of the translates of $-\sigma$ cover $N_{\mathbb{R}}$ and so 
the corresponding sets cover $X_{\sigma}$. More precisely, for any point $z$ in $X_{\sigma}$, there
is a translate of $-\sigma$ which corresponds to a neighborhood of $z$.

We consider first the complex singularity exponent of $h$ at $z_{\sigma}$. Suppose
$c>0$ is such that $e^{-2c\psi}$ is integrable in a neighborhood of $z_{\sigma}$.
Then it is integrable in a neighborhood $C_y$.
We have first that,
$$\int_{C_y} e^{-2c\psi(z)}dz\wedge d\overline{z} = \int_{T_c \times (y-\sigma)} e^{-2c\psi(z)}dz\wedge d\overline{z}.$$ 
Recall that $\psi(z)=\langle -v_{\sigma},\ln|z| \rangle + \varphi(z)$, and that $f$ is the function on $N_{\mathbb{R}}$ such that $f(x)=\varphi(e^x)$.

Say we have chosen a basis of $N$ or equivalently of $M$, and we denote by $(x_i)_{i=1\ldots n}$ the 
coordinates of $x\in N_{\mathbb{R}}$ along this basis. This determines local holomorphic coordinates
$z_i= e^{x_i+i \theta_i}$ on $T \simeq N_{\mathbb{R}}\times T_c$. 
Using the fact that 
$\frac{dz_i}{z_i} \wedge \frac{d\overline{z_i}}{\overline{z_i}}= dx_i \wedge d\theta_i$, 
and $T_c$-invariance, we obtain that, up to a constant,                           
$$\int_{C_y} e^{-2c\psi(z)}dz\wedge d\overline{z} = 
\int_{y-\sigma} e^{-2c(f(x)+\langle -v_{\sigma},x \rangle)} e^{2\sum_i x_i}dx.$$

Since $\sum_i x_i$ is equal to $\langle S_{\sigma}^*(1,\ldots ,1),x \rangle$, 
we conclude by using Proposition~\ref{Gue}
that the complex singularity exponent $c_{z_{\sigma}}(h)$ is the supremum of the $c>0$ such that 
$0\in \mathrm{Int}(N_{-\sigma}(2c(f+\langle -v_{\sigma},\cdot  \rangle)-2\langle  S_{\sigma}^*(1,\ldots, 1),\cdot  \rangle))$.

To obtain a simpler condition, remark that for any function $g$ and positive scalar $\lambda$, 
$N_{-\sigma}(\lambda g)=\lambda N_{-\sigma}(g)$, and that if $g_1$ and $g_2$ are two convex functions then $N_{-\sigma}(g_1+g_2)$ is 
the Minkowski sum of $N_{-\sigma}(g_1)$ and $N_{-\sigma}(g_2)$.

So we get 
$c_{z_{\sigma}}(h)=\mathrm{sup}\{c>0| cv_{\sigma} \in \mathrm{Int}(N_{-\sigma}(cf))-S_{\sigma}^*(1,\ldots,1)\}.$

Furthermore, for any $c<c_{z_{\sigma}}(h)$, the Proposition~\ref{Gue} shows that  $e^{-2c\psi}$ is
integrable on every $C_y$ for $y\in  N_{\mathbb{R}}$.
Observe now that for any point $z\in X_{\sigma}$, there exists a $C_y$ containing $z$. 
So for any point $z\in X_{\sigma}$, $c_{z}(h)\geq c_{z_{\sigma}}(h)$.
This concludes the proof of the proposition.
\end{proof}

\subsection{lct on a compact smooth toric manifold}

\begin{thm}
\label{calclct}
Let $X_{\Sigma}$ be a smooth compact toric manifold, $L$ a linearized line bundle 
on $X_{\Sigma}$ and $h$ a $T_c$-invariant non-negatively curved metric on $L$.
Then 
$$\mathrm{lct}(h)=\mathrm{sup}\{c>0|cP_L\subset \mathrm{Int}(cP_h+P_{-K_{X_{\Sigma}}}) \}.$$\end{thm}

\begin{proof}
The compact manifold $X_{\Sigma}$ is covered by the affine toric manifolds $X_{\sigma}$, for 
$\sigma \in \Sigma(n)$.
By definition of the log canonical threshold, 
$$\mathrm{lct}(h)=\mathrm{min}_{\sigma \in \Sigma(n)}\mathrm{lct}(h|_{Z_{\sigma}}).$$
Another way to say this is that $\mathrm{lct}(h)$ is the sup of $c>0$ such that 
$c \leq \mathrm{lct}(h|_{X_{\sigma}})$ for all $\sigma \in \Sigma(n)$.

Now this condition means, by Proposition~\ref{lctaff}, that for all $\sigma \in \Sigma(n)$,
$$cv_{\sigma} \in \mathrm{Int}(N_{-\sigma}(cf_h+\langle -S_{\sigma}^*(1,\ldots,1),\cdot \rangle).$$

By Proposition~\ref{propN}, this is equivalent to the condition that for all $\sigma \in \Sigma(n)$,
$$cv_{\sigma} + \sigma^{\vee} \subset \mathrm{Int}(N_{-\sigma}(cf_h+\langle -S_{\sigma}^*(1,\ldots,1),\cdot \rangle).$$

This is further equivalent to the condition that for all $\sigma \in \Sigma(n)$,    %maybe explain why this is equivalent:
%v_{\sigma}\in P_L \subset v_{\sigma} + \sigma^{\vee}
$$\bigcap_{\sigma\in \Sigma(n)}(cv_{\sigma}+\sigma^{\vee}) \subset \mathrm{Int}(N_{-\sigma}(cf_h+\langle -S_{\sigma}^*(1,\ldots,1),\cdot \rangle).$$

Recall from Proposition~\ref{Polytope}
that $\bigcap_{\sigma\in \Sigma(n)}(v_{\sigma}+\sigma^{\vee}) = N(f_L)= P_L$, so that the condition can be written: 
$$N(cf_L)\subset \bigcap_{\sigma \in \Sigma(n)}  \mathrm{Int}(N_{-\sigma}(cf_h+\langle -S_{\sigma}^*(1,\ldots,1),\cdot \rangle)= \mathrm{Int}(N(cf_h+f_{-K_{X_{\Sigma}}})).$$
Indeed, the support function of the anticanonical bundle is, from Example~\ref{anticanonical}, 
$$f_{-K_{X_{\Sigma}}}(x)=\langle -S_{\sigma}^*(1,\ldots,1),x \rangle.$$
\end{proof}

\section{Alpha-invariant}

\subsection{Log canonical threshold and $\alpha$-invariant}

Let $X$ be a compact Kähler manifold, $L$ a big and nef line bundle on $X$.

\begin{defn}
Assume that a compact group $K$ acts on $X$, and that $L$ is $K$-linearized.
The alpha invariant $\alpha_K(L)$ of $L$ with respect to the group $K$ is defined as the infimum 
of the log canonical thresholds of all $K$-invariant, non negatively curved 
singular hermitian metrics on $L$. 
\end{defn}

The linear systems in a multiple of $L$ give singular metrics on $L$, that we will 
call algebraic metrics, in the following way.
Let $\delta_1,\ldots, \delta_r \in H^0(X,mL)$ be linearly independent sections, and 
denote by $\Delta$  the linear system generated by these.
Then it defines an algebraic metric $h_{\Delta/m}$ on $L$ by setting, in any trivialization,
$$||\xi||^2_{h_{\Delta/m}}=\frac{|\xi|^2}{(\sum |\delta_j(z)|^2)^{1/m}},$$
for any $\xi \in L_z$.
The local potential $\varphi_{\Delta/m}(z)=\frac{1}{2m}\mathrm{ln}\sum|\delta_j(z)|^2$ is psh.

If $\Delta$ is one dimensional, generated by $\delta$, we denote by $h_{\delta/m}$ the corresponding metric.

Recall the following result of Demailly, relating the $\alpha$-invariant with log canonical thresholds
of algebraic metrics:
\begin{thm}
\cite[Appendix A]{CSD08}
Let $K$ be a compact group, let $X$ be a compact complex $K$-variety and $L$ a big and nef
$K$-linearized line bundle on $X$. Then $$\alpha_K(L)=\mathrm{inf}_{m\in \mathbb{N}^*}\mathrm{inf}_{\Delta \subset H^0(X,mL),~ \Delta^K=\Delta} \mathrm{lct}(h_{\Delta/m}).$$
\end{thm}

One can slightly improve this result, and give the following statement, which is only given in
the case of a trivial group $K$ by Demailly.

\begin{cor}
\label{metirr}
Let $K$ be a compact group, let $X$ be a compact complex $K$-variety and $L$ a big and nef
$K$-linearized line bundle on $X$. Then 
$$\alpha_K(L)=\mathrm{inf}_{m\in \mathbb{N}^*}\mathrm{inf}_{\Delta \in \mathrm{Irr}(H^0(X,mL))} \mathrm{lct}(h_{\Delta/m}),$$
where $\mathrm{Irr}(H^0(X,mL))$ denotes the set of all irreducible $K$-subrepresentations of $H^0(X,mL)$.
\end{cor}

\begin{proof}
Let $\Delta$ be a $K$-invariant subspace of $H^0(X,mL)$, 
then $\Delta=\Delta_1\oplus \cdots \oplus \Delta_s$ with $\Delta_i$ irreducible subspaces.
For all $i$, one can choose a basis $\delta_j^i$ of $\Delta_i$. 
Together they form a basis of $\Delta$ and we can obtain the metric $h_{\Delta}$ this way.

In particular, $\varphi_{\Delta/m}(z)=\frac{1}{2m}\mathrm{ln}\sum_i\sum_j|\delta^i_j(z)|^2$.
Since the logarithm is increasing we can write
$$\varphi_{\Delta/m}(z)\geq \frac{1}{2m}\mathrm{ln}\sum|\delta^1_j(z)|^2 =\varphi_{\Delta_1/m}(z).$$

This implies, by elementary properties of the complex singularity exponent, \cite[1.4]{DK01}
that $\mathrm{lct}(h_{\Delta/m})\geq \mathrm{lct}(h_{\Delta_1/m})$.

We conclude that the log canonical threshold of a metric associated to a $K$-invariant 
linear system is greater than the log canonical threshold of at least one metric associated to 
an irreducible linear system, so it is enough to consider only these.
\end{proof}

\subsection{General formula}

Let $X_{\Sigma}$ be a smooth compact toric manifold.
Let $N(T)$ be the normalizer of $T$ in $\mathrm{Aut}(X_{\Sigma})$, and denote by $W=N(T)/T$ the 
Weyl group obtained from $T$.

The group $N(T)$ naturally acts on $M$ and since $T$ acts trivially on $M$, this induces
an action of $W$ on $M$. By duality one also gets an action on $N$.

From the description of morphisms between toric varieties \cite[Theorem 3.3.4]{CLS11},
we can see that $W$ is 
isomorphic to the subgroup of $\mathrm{GL}(N)$ composed of the $\rho$ such that 
$\rho(\Sigma)=\Sigma$. In particular, $W$ is finite.

Given a subgroup $G$ of $W$, 
we denote by $T_G$ the preimage in $N(T)$ of $G$, and let $K_G:=K\cap T_G$.
If $P$ is a polytope in $M_{\mathbb{R}}$ we let $P^G$ be the set 
of $G$-invariant points of $P$. 

Finally, if $P$ is a polytope in $M_{\mathbb{R}}$, we denote by $P(\mathbb{Q})$
the set of rational points in $P$, i.e. points $p$ such that there exists $m\in \mathbb{N}^*$
with $mp\in M$.

\begin{thm}
\label{Galpha}
Let $L$ be a $T_G$-linearized line bundle on $X_{\Sigma}$. Then
$$\alpha_{K_G}(L)=\mathrm{inf}_{p\in P_L^G(\mathbb{Q})}\mathrm{sup}\{c>0|cP_L\subset \mathrm{Int}(cp+P_{-K_{X_{\Sigma}}})\}.$$
\end{thm}

\begin{proof}
The Corollary~\ref{metirr} shows that it is enough to consider algebraic metrics on $L$
associated to $K_G$-irreducible linear system in a multiple of $L$.

The $T_c$-irreducible subrepresentations of $H^0(X_{\Sigma},mL)$ are the dimension one 
subspaces corresponding to integral points of the polytope $P_{mL}$ associated to $mL$.
Recall that $P_{mL}=mP_L$.

Now a $K_G$-irreducible subrepresentation of $H^0(X_{\Sigma},mL)$ is the union of 
the images by $G$ of a $T_c$-irreducible representation.

Let $p$ be an integral point in $mP_L$, and denote by $\Delta$ the $K_G$-irreducible
linear system generated by the $G$-orbit of $p$.

The potential of $h_{\Delta/m}$ is
$$\varphi_{\Delta/m}(z)=\frac{1}{2m}\mathrm{ln}\left(\sum_{g\in G}|(g\cdot p)(z)|^2\right).$$

By arithmetico-geometric inequality, 
$$\varphi_{\Delta/m}(z)\geq \frac{1}{2m}\mathrm{ln}\left|\left(\frac{\sum_{g\in G}(g\cdot p)}{|G|}\right)(z)\right|^2.$$

The right-hand side of this inequality is the potential of the algebraic metric
$h_{\frac{\sum_{g\in G}(g\cdot p)}{m|G|}}$ corresponding to the linear system 
of $H^0(X_{\Sigma},m|G|L)$ generated by the section $\sum_{g\in G}(g\cdot p)$.

Using again the fact that the complex singularity exponent is increasing \cite[1.4]{DK01}, 
we get 
$$\mathrm{lct}(h_{\Delta/m})\geq \mathrm{lct}(h_{\frac{\sum_{g\in G}(g\cdot p)}{m|G|}}).$$

We have thus shown that it is enough to compute the log canonical thresholds 
of algebraic metrics associated to one dimensional $G$-invariant sublinear 
systems of multiples of $L$.

We use Theorem~\ref{calclct} to conclude.
Indeed if $p\in mP_L$ generates a one dimensional $G$-invariant sublinear 
system in $H^0(X_{\Sigma},mL)$, and $f_{p/m}$ denotes the convex function associated 
to the potential of the corresponding algebraic metric $h_{p/m}$, we have $N(f_{p/m})=\{p/m\}$.
 
Applying Theorem~\ref{calclct} gives 
$$\mathrm{lct}(h_{p/m})=\sup\{c>0|cP_L\subset \mathrm{Int}(cp/m+P_{-K_{X_{\Sigma}}})\}.$$

Finally, observe that as $p$ and $m$ vary, they describe the set $P_L^G(\mathbb{Q})$
of $G$-invariant points of $P_L$ with rational coordinates.
\end{proof}

\begin{rem}
One can also prove, without the use of Corollary~\ref{metirr}, that we can consider only 
metrics corresponding to points of $P_L$ (not necessarily with rational coordinates), 
by considering the expression of the log canonical threshold of any metric.

Indeed, if $f$ is a convex function on $N_{\mathbb{R}}$, corresponding to a metric $h$ on $L$, and 
$p$ is a point in $N(f)$, then the metric $h_p$ associated to the convex function 
$x\mapsto \langle p,x \rangle$ is also a non-negatively curved metric on $L$, and $\mathrm{lct}(h_p)\leq \mathrm{lct}(h)$.
\end{rem}

\subsection{Case of the anticanonical line bundle}

We assume in this section that $L=-K_{X_{\Sigma}}$.

This line bundle admits a natural $\mathrm{Aut}(X)$-linearization, 
and the polytope associated to this 
linearization contains the origin in its interior, because $-K_X$ is big.

For any subgroup $G$ of $W$, let $S_G:=\{p\in \partial P_L|g\cdot p = p ~\forall g\in G\}$.
If $0\neq p\in P_L$, let $w_p$ be the point $\partial P_L \cap \{-tp|t\geq0\}$.

\begin{rem} \mbox{}
\begin{itemize}
\item $S_G$ is empty if and only if $\{0\}$ is the only point fixed by $G$ in $P$.
\item If $S_W$ is empty, $X_{\Sigma}$ is called symmetric. 
\end{itemize}
\end{rem}

\begin{prop}
\label{lctanti}
Assume that $P_h=\{p\}$ with $0\neq p\in P_L$. Then 
$$\mathrm{lct}(h)=\frac{|w_p|}{|w_p|+|p|}.$$
\end{prop}

\begin{figure}
\begin{tikzpicture}[scale=1]
\draw[black, fill=black] (3,0) circle (0.07);
\draw[black, fill=black] (0,0) circle (0.07);
\draw[black, fill=black] (0,3) circle (0.07);
\draw[black, fill=black] (1,1) circle (0.07);
\draw[black, fill=black] (0.5,0.5) circle (0.07);
\draw[black, fill=black] (1.5,1.5) circle (0.07);

\draw (0.7,0.4) node {$p$};
\draw (0.9,1.2) node {0};
\draw (1.6,1.7) node {$w_p$};

\draw[] (0,0) -- (3,0);
\draw[] (0,0) -- (0,3);
\draw[] (3,0) -- (0,3);
\draw[] (0.5,0.5) -- (1.5,1.5);
\end{tikzpicture} 
\end{figure}

\begin{proof}
By Theorem~\ref{calclct} we have
$$\mathrm{lct}(h)=\sup\{c>0| cP\subset \mathrm{Int}(cp+P)\}.$$  

Consider the half-line starting from $p$ and containing the origin.
It intersects $\partial P$ at $w_p$.
Denote by $r$ its intersection with $\partial (p+P)$.

Then it is easy to see that the log canonical threshold of $h_p$ 
is equal to the quotient of the distance between $p$ and $r$ by the 
distance between $p$ and $w_p$.
The translation sending $0$ to $p$ also sends $w_p$ to $r$, so 
$|r-p|=|w_p|$.
The result follows.
\end{proof}

\begin{rem}
\label{h0}
If $P_h=\{0\}$ then $\mathrm{lct}(h)=1$.
\end{rem}

\begin{exa}
Consider the case $P_h=\{b\}$, where $b$ is the barycenter of the polytope $P_L$.
Then $\mathrm{lct}(h)$ is equal to the greatest lower bound for Ricci curvature $R(X)$,
introduced by Sz{\'e}kelyhidi \cite{Sze11}, and computed for toric manifolds by Li \cite{Li11}. 
\end{exa}

From this formula we recover the previous results of Song and Chel'tsov-Shramov. 

\begin{thm}
\cite{Son05} \cite[Lemma 6.1]{CSD08}
\label{previous}
Let $X$ be a smooth Fano toric manifold, and $G$ be a subgroup of $W$. Then 
\begin{itemize}
\item if $S_G$ is empty, $\alpha_{K_G}(X)=1$;
\item else, $\alpha_{K_G}(X)=\frac{1}{1+\mathrm{max}_{p\in S_G}\frac{|p|}{|w_p|}}\leq\frac{1}{2}$.
\end{itemize}
\end{thm}

\begin{proof}
By Theorem~\ref{Galpha}, it is enough to consider only the (rational)
$G$-invariant points of $P$. 

The first case follows immediately using Remark~\ref{h0}.

In the second case, we obtain the formula using Proposition~\ref{lctanti}.
Indeed, it is enough to consider points $p$ in $S_G$ because 
if $q\neq 0$ is not in $\partial P$, and $p$ 
is the intersection of $\partial P$ with the half line
starting from the origin and going through $q$, then 
$\mathrm{lct}(h_q)\geq \mathrm{lct}(h_p)$.

Furthermore, $\max_{p\in S_G}\frac{|p|}{|w_p|}\geq 1$ 
because otherwise if $p$ was such a point at which this maximum was attained 
and it was $<1$ then we would have $\frac{|w_p|}{|p|}>1$ with $w_p\in S_G$, 
which is a contradiction.
\end{proof}

\subsection{Example}

We compute the $\alpha$-invariant of any linearized line bundle on the blow up $X$ of 
$\mathbb{P}^2$ at one point which we denote $X$ in the following.

Identify $N$ with $\mathbb{Z}^2$.
The fan of $X$ has four rays, with generators 
$u_1=(1,0)$, $u_2=(1,1)$, $u_3=(0,1)$ and $u_4=(-1,-1)$.

The group $W$ is isomorphic to $\mathbb{Z}/2\mathbb{Z}$ and acts on $M_{\mathbb{R}}$ 
by exchanging the coordinates $(x,y)\mapsto (y,x)$.

We define the polytope $P(k,l)$ to be the polytope whose vertices are 
$(0,k)$, $(0,l)$, $(k,0)$ and $(l,0)$, for $k,l \in \mathbb{N}$ with $l>k$.
It is easy to see that the polytopes of nef and big divisors are the $P(k,l)$, up to 
translation by a character.
For example, the polytope of the anticanonical bundle is $Q:=(-1,-1)+P(1,3)$.

\begin{prop}
The $\alpha$-invariant with respect to $K_W$
of the nef and big line bundle corresponding to $P(k,l)$ is equal to 
$\mathrm{inf}(\frac{1}{l-k},\frac{2}{l})$. 
\end{prop}

\begin{proof}
By Theorem~\ref{Galpha}, it is enough to consider points (with rational coordinates)
in the intersection of $P(k,l)$ with the first diagonal. However, one easily remarks 
that it is enough to consider only the point $(l/2,l/2)$, similarly to the proof of Theorem~\ref{previous}.

We want to compute 
$$\mathrm{sup}\{c>0|cP(k,l)\subset \mathrm{Int}(c(l/2,l/2)+Q)\}.$$
This is of course equal to 
$$\mathrm{sup}\{c>0|P(k,l)\subset \mathrm{Int}((l/2,l/2)+\frac{1}{c}Q)\}.$$

Observe that $l/2$ is the least positive constant $b$ such that 
$$\{(0,l), (l,0)\} \subset (l/2,l/2)+bQ.$$
If $k\geq l/2$, then we have also $\{(0,k),(k,0)\}\subset (l/2,l/2)+l/2Q$, 
so 
$$P(k,l)\subset (l/2,l/2)+l/2Q.$$
Thus $\alpha_{K_W}(P(k,l))=2/l$ when $k\geq l/2$.

For the other case, observe that $l-k$ is the least positive constant $b$
such that 
$(k/2,k/2)\in (l/2,l/2)+bQ$.
If $k\leq l/2$, then we have also 
$$P(k,l)\subset (l/2,l/2)+(l-k)Q.$$
Thus $\alpha_{K_W}(P(k,l))=\frac{1}{l-k}$ when $k\geq l/2$.
\end{proof}

\bibliographystyle{alpha}
\bibliography{biblio}

\end{document}